\documentclass[11pt]{article}%
\usepackage{mathrsfs}
\usepackage{bbm}
\usepackage{amsfonts}
\usepackage{amsmath,amssymb}
\usepackage{amsmath}
\usepackage{amssymb}
\usepackage{graphicx}%
\usepackage{shorttoc}
\providecommand{\U}[1]{\protect \rule{.1in}{.1in}}

\newtheorem{theorem}{Theorem}[section]
\newtheorem{corollary}[theorem]{Corollary}
\newtheorem{definition}[theorem]{Definition}

\newtheorem{lemma}[theorem]{Lemma}

\newtheorem{remark}[theorem]{Remark}

\newenvironment{proof}[1][Proof]{\noindent \textbf{#1.} }{\  $\Box$}
\numberwithin{equation}{section}

\begin{document}

\title{\textbf{Normal Approximation by Stein's Method under Sublinear Expectations }}
\author{ Yongsheng Song\thanks{%
Academy of Mathematics and Systems Science, CAS, Beijing, China,
yssong@amss.ac.cn. Research  supported  by NCMIS;
Key Project of NSF (No. 11231005); Key Lab of Random Complex
Structures and Data Science, CAS (No. 2008DP173182) and Key Research Program of Frontier Sciences, CAS.}  }

\date{\today}
\maketitle

\begin{abstract}Peng (2008)(\cite{P08b}) proved the Central Limit Theorem under a sublinear expectation:

\textit{Let $(X_i)_{i\ge 1}$ be a sequence of i.i.d random variables  under a sublinear expectation $\hat{\mathbf{E}}$ with $\hat{\mathbf{E}}[X_1]=\hat{\mathbf{E}}[-X_1]=0$ and $\hat{\mathbf{E}}[|X_1|^3]<\infty$. Setting $W_n:=\frac{X_1+\cdots+X_n}{\sqrt{n}}$, we have, for each bounded and Lipschitz function $\varphi$,
\[\lim_{n\rightarrow\infty}\bigg|\hat{\mathbf{E}}[\varphi(W_n)]-\mathcal{N}_G(\varphi)\bigg|=0,\] where $\mathcal{N}_G$ is the $G$-normal distribution with $G(a)=\frac{1}{2}\hat{\mathbf{E}}[aX_1^2]$, $a\in \mathbb{R}$.}

 In this paper, we shall give an estimate of the rate of convergence of this CLT by Stein's method under sublinear expectations:

 \textit{Under the same conditions as above, there exists $\alpha\in(0,1)$  depending  on $\underline{\sigma}$ and $\overline{\sigma}$, and a positive constant $C_{\alpha, G}$ depending  on $\alpha, \underline{\sigma}$ and $\overline{\sigma}$ such that
\[\sup_{|\varphi|_{Lip}\le1}\bigg|\hat{\mathbf{E}}[\varphi(W_n)]-\mathcal{N}_G(\varphi)\bigg|\leq C_{\alpha,G}\frac{\hat{\mathbf{E}}[|X_1|^{2+\alpha}]}{n^{\frac{\alpha}{2}}},\] where $\overline{\sigma}^2=\hat{\mathbf{E}}[X_1^2]$, $\underline{\sigma}^2=-\hat{\mathbf{E}}[-X_1^2]>0$ and $\mathcal{N}_G$ is the $G$-normal distribution with
\[G(a)=\frac{1}{2}\hat{\mathbf{E}}[aX_1^2]=\frac{1}{2}(\overline{\sigma}^2a^+-\underline{\sigma}^2a^-), \ a\in \mathbb{R}.\]}

\end{abstract}


\textbf{MSC-classification}: 60F05; 60G50

\section{Introduction}
The Central Limit Theorem is one of the most striking and useful results in probability and statistics, and explains why the normal distribution appears in areas as diverse as gambling, measurement error, sampling, and statistical mechanics. In essence, the Central Limit Theorem in its classical form states that a normal approximation applies to the distribution of quantities that can be modeled as the sum of many independent contributions, all of which are roughly the same size.

 Motivated by problems of model uncertainty in statistics, measures of risk and
superhedging in finance, Peng (2007) introduced the notion of sublinear expectations. A random variable $X$ in a sublinear  expectation space $(\Omega, \mathcal {H}, \hat{\mathbf{E}})$ with $\hat{\mathbf{E}}[|X|^3]<\infty$ is called $G$-normally distributed if for any independent copy $X'$ of $X$ and $\alpha, \beta\in \mathbb{R}$, \[ \alpha X+\beta X' \mathop{= }^d \sqrt{\alpha^2+\beta^2}X.\]  As is known, if $\hat{\mathbf{E}}$ is a linear expectation generated by a probability, a random variable $X$ with the above property is normally distributed. Suppose $X$ is $G$-normally distributed under $\hat{\mathbf{E}}$. For $\varphi\in C_{b,Lip}(\mathbb{R})$, the collection of bounded Lipstchiz functions on $\mathbb{R}$, set $\mathcal{N}_G[\varphi]=\hat{\mathbf{E}}[\varphi(X)]$. We call $\mathcal{N}_G$, a sublinear expectation on $C_{b,Lip}(\mathbb{R})$, a $G$-normal distribution. Here, the function $G$, defined by $G(a)=\frac{1}{2}\hat{\mathbf{E}}[aX^2]$, $a\in\mathbb{R}$, characterizes the variances of $X$.

Peng (2008) proved the Central Limit Theorem under a sublinear expectation.
\begin{theorem}\label {intro-PengCLT} Let $(X_i)_{i\ge 1}$ be a sequence of i.i.d random variables  under a sublinear expectation $\hat{\mathbf{E}}$ with $\hat{\mathbf{E}}[X_1]=\hat{\mathbf{E}}[-X_1]=0$ and $\hat{\mathbf{E}}[|X_1|^3]<\infty$. Setting $W_n:=\frac{X_1+\cdots+X_n}{\sqrt{n}}$, we have, for each  $\varphi\in C_{b,Lip}(\mathbb{R})$,
\[\lim_{n\rightarrow\infty}\bigg|\hat{\mathbf{E}}[\varphi(W_n)]-\mathcal{N}_G(\varphi)\bigg|=0,\] where $\mathcal{N}_G$ is the $G$-normal distribution with $G(a)=\frac{1}{2}\hat{\mathbf{E}}[aX_1^2]$, $a\in \mathbb{R}$.
\end{theorem}

Just like the linear case, this theorem mathematically justified, at least asymptotically,  the $G$-normal distribution
may be used to approximate quantities which can be formulated as the sums of independent and identically distributed random variables under a sublinear expectation. However,
even though in practice sample sizes may be large, or may appear to be sufficient for
the purposes at hand, depending on that and other factors, the normal approximation
may or may not be accurate. It is here the need for the evaluation of the quality of
the normal approximation arises.

For the linear case,  Stein's method, which made its first appearance in the ground
breaking work of Stein (1972), is a powerful tool to estimate  the error of normal approximation. The cornerstone of  Stein's method is the Stein equation (refer to \cite{CGS11} for more details): For a standard normally distributed random variable $Z$ and given $\varphi$, solve the following equation for $f$,
\begin {eqnarray}\label {Se}f'(x)-xf(x)=\varphi(x)-E[\varphi(Z)].
\end {eqnarray}
Then, for any random variable $W$, evaluate the left hand side of the Stein equation at $W$ and
take the expectation, obtaining $E[\varphi(W)]-E[\varphi(Z)]$.

The job of this paper is to introduce the ideas of Stein's method to the nonlinear case. The expected Stein equation for $G$-normal distribution would be
\begin {eqnarray}\label {GSe}G(f''(x))-\frac{x}{2}f'(x)=\varphi(x)-\mathcal{N}_G[\varphi].
\end {eqnarray}
Unfortunately, for $\varphi\in C_{b,Lip}(\mathbb{R})$, Equ. (\ref{GSe}) generally does not have a solution. Therefore, the first step is to find a substitute of the Stein equation.

For $\varphi\in C_{b,Lip}(R)$, the function $u(x,t):=\mathcal{N}_G[\varphi(x+\sqrt{t}\cdot)]$ is the unique viscosity solutions of the $G$-heat equation below
\begin {eqnarray*}
\partial_t u-G(D^2_x u)&=&0, \ (x,t)\in \mathbb{R}\times(0,\infty),\\
                        u(x,0)&=& \varphi (x),
\end {eqnarray*} where $G(a)=\frac{1}{2}\mathcal{N}_G[ax^2]$, $a\in \mathbb{R}$, is  determined by the variances $\overline{\sigma}^2:=\mathcal{N}_G[x^2]$ and  $\underline{\sigma}^2:=-\mathcal{N}_G[-x^2]$. So, if $\overline{\sigma}=\underline{\sigma}=\sigma$, $\mathcal{N}_G$ is nothing but the classical normal distribution $N(0,\sigma^2)$.

 Let $\Theta$ be a weakly compact subset  of probability measures  on $(%
\mathbb{R},\mathcal{B}(\mathbb{R}))$. For the sublinear expectation $\mathcal{N}[\varphi]=\sup_{\mu\in\Theta}\mu[\varphi]$
 on $C_{b,Lip}(\mathbb{R})$ and  a function $\phi\in C_{b,Lip}(\mathbb{R})$, set \[w(t)=\mathcal{N}[v(\sqrt{1-t}\cdot,t)],\] where $v$ is the solution to the $G$-heat equation with initial value $\phi$.
Then $w(1)=\mathcal{N}_G[\phi]$, $w(0)=\mathcal{N}[\phi]$, and it can be shown that, for a.e. $s\in(0,1)$,
\begin {eqnarray}\label {intro-derivative}
w^{\prime}(s)=\frac{1}{1-s}\mu_s[G(\phi''_s(x))-\frac{1}{2}x\phi'_s(x)],
\end {eqnarray} where $\phi_s(x)=v(\sqrt{1-s}x,s)$ and $\mu_s\in\Theta$ with $\mu_s[\phi_s]=\mathcal{N}[\phi_s]$.  From this, we get a substitute of the Stein equation.

\textbf{Step 1.} $\mathcal{N}_G[\phi]-\mathcal{N}[\phi]=\int_0^1\frac{1}{1-s}\mu_s[G(\phi''_s(x))-\frac{1}{2}x\phi'_s(x)] ds$.

Return to the linear case,  i.e.,  $\underline{\sigma}=\overline{\sigma}$ and $\Theta=\{\mu\}$ is a singleton, the above formula will reduce to the classical Stein equation (see Remark \ref {remark-SteinE} for details).

Now the next job is to calculate the expectation on the right side of the equality (\ref {intro-derivative}).

Let $\alpha\in(0,1)$. Suppose $\mathcal{N}[x]=\mathcal{N}[-x]=0$ and $\mathcal{N}[|x|^{2+\alpha}]<\infty$. For $\phi\in C_b^{2,\alpha}(\mathbb{R})$ and $\mu\in\Theta$ with $\mu[\phi]=\mathcal{N}[\phi]$, we have

\textbf{Step 2.} $\bigg|\mu[G(\phi''(x))-\frac{1}{2}x\phi'(x)]\bigg|\le 2 [\phi'']_{\alpha} \mathcal{N}[|x|^{2+\alpha}],$ where $G(a)=\frac{1}{2}\mathcal{N}[a|x|^2]$, $a\in \mathbb{R}$.

It merits to emphasize that the function $G$ in Step 1 is determined by the variances of $\mathcal{N}_G$ and that the function $G$ in Step 2 is determined by the variances of $\mathcal{N}$. In other words, to estimate $\mathcal{N}_G[\phi]-\mathcal{N}[\phi]$ applying Step 1 and Step 2 requires that  $\mathcal{N}_G$ and $\mathcal{N}$ have the same variances.

Besides, note that $\phi_s(x)$ in the equality (\ref {intro-derivative}) is the solution to the $G$-heat equation. Therefore, to apply the estimate in Step 2, we need the regularity properties of the $G$-heat equation, which can be found in the literatures of the partial differential equations (see Section 3 for details).

\textbf{Step 3.}  $[D_x^2v(\cdot,t)]_{\alpha}\le c_{\alpha, G}\frac{1}{t^{\frac{1}{2}+\frac{\alpha}{2}}}|\phi|_{Lip}$ for some $\alpha\in(0,1)$ and $c_{\alpha,G}>0$.

Following these three steps, we  give the estimates of the rate of convergence of Peng's Central Limit Theorem.

\emph{Under the same conditions as those in Theorem \ref{intro-PengCLT}, there exists $\alpha\in(0,1)$  depending  on $\underline{\sigma}$ and $\overline{\sigma}$, and a positive constant $C_{\alpha, G}$ depending  on $\alpha, \underline{\sigma}$ and $\overline{\sigma}$ such that
\[\sup_{|\varphi|_{Lip}\le1}\bigg|\hat{\mathbf{E}}[\varphi(\frac{X_1+\cdots+X_n}{\sqrt{n}})]-\mathcal{N}_G(\varphi)\bigg|\leq C_{\alpha,G}\frac{\hat{\mathbf{E}}[|X_1|^{2+\alpha}]}{n^{\frac{\alpha}{2}}}ïŒ\] where $\overline{\sigma}^2=\hat{\mathbf{E}}[X_1^2]$, $\underline{\sigma}^2=-\hat{\mathbf{E}}[-X_1^2]>0$ and $\mathcal{N}_G$ is the $G$-normal distribution with
\[G(a)=\frac{1}{2}\hat{\mathbf{E}}[aX_1^2]=\frac{1}{2}(\overline{\sigma}^2a^+-\underline{\sigma}^2a^-), \ a\in \mathbb{R}.\]
}

Here $\alpha$ is the H$\ddot{\textmd{o}}$lder exponent in Step 3, and $C_{\alpha,G}$ can be chosen as $\frac{4 }{1-\alpha}c_{\alpha,G}$
with $c_{\alpha,G}$ the $\alpha$-H$\ddot{\textmd{o}}$lder constant in Step 3.

In Section 2, we review the basic notions and results of sublinear expectations. In Section 3, we introduce the regularity properties of the $G$-heat equation that will be used in this paper. In Section 4, we shall generalize the idea of Stein's method to the sublinear expectation space, based on which we get the rate of convergence of Peng's Central Limit Theorem. In Section 5, we consider the CLT under sublinear expectations of a sequence of independent random variables which may not be identically distributed.
\section{Basic Notions of Sublinear Expectations}
Here we review basic notions and results of sublinear expectations.  The readers may refer to \cite{P07a}-\cite{P10} for more details.

Let $\Omega$ be a given set and let $\mathcal{H}$ be a  linear space of real valued functions defined on $\Omega$ such that for any $X\in\mathcal{H}$ and $\varphi\in C_{b,Lip}(\mathbb{R})$, we have $\varphi(X)\in \mathcal{H}$.
The space $\mathcal{H}$ is considered as our space of random variables.

\begin{definition}
A  sublinear expectation  is a functional $\hat{\mathbf{E}}: \mathcal{H}\to \mathbb{R}$ satisfying
\begin{description}
\item[E1.]  $\hat{\mathbf{E}}[X]\geq \hat{\mathbf{E}}[Y],\ \text{if}\ X\ge Y$;

\item[E2.]  $\hat{\mathbf{E}}[\lambda X]=\lambda \hat{\mathbf{E}}[X],\ \text{for}\ \lambda\geq 0$;

\item[E3.]  $\hat{\mathbf{E}}[c]=c, \ \text{for}\ c\in \mathbb{R}$;

\item[E4.]  $\hat{\mathbf{E}}[X+Y]\leq \hat{\mathbf{E}}[X]+\hat{\mathbf{E}}[Y]$, \ for  $X, Y\in \mathcal{H}$;

\item[E5.]  $\hat{\mathbf{E}}[\varphi_n(X)]\downarrow0$, for $X\in\mathcal{H}$ and $\varphi_n\in C_{b,Lip}(\mathbb{R})$, $\varphi_n\downarrow0$.
\end{description}
\end{definition}
The triple $(\Omega, \mathcal{H}, \hat{\mathbf{E}})$ is called a  sublinear expectation space. For $X\in\mathcal{H}$, set \[\mathcal{N}^X[\varphi]=\hat{\mathbf{E}}[\varphi(X)],  \ \varphi\in C_{b,Lip}(\mathbb{R}),\]  which is a sublinear expectation on $C_{b,Lip}(\mathbb{R})$. We say $X$ is distributed as  $\mathcal{N}^X$, write $X\sim\mathcal{N}^X$.  A functional $\mathcal{N}$ is a sublinear expectation  on $C_{b,Lip}(\mathbb{R})$ if and only if it can be represented as the
supremum expectation of a weakly compact subset $\Theta$ of probability measures  on $(%
\mathbb{R},\mathcal{B}(\mathbb{R}))$ (see \cite{DHP11}),
\begin {eqnarray}
\mathcal{N}[\varphi]=\sup_{\mu\in\Theta}\mu[\varphi], \ \textmd{for all} \  \varphi\in C_{b,Lip}(\mathbb{R}).
\end {eqnarray}

\begin {definition} Let  $(\Omega, \mathcal{H}, \hat{\mathbf{E}})$ be a sublinear expectation space. We say a random vector $\mathbf{ X}=(X_1,\cdots, X_m)\in\mathcal{H}^m$ is independent from $\mathbf{ Y}=(Y_1,\cdots, Y_n)\in\mathcal{H}^n$ if for any $\varphi\in C_{b,Lip}(\mathbb{R}^{n+m})$
\[\hat{\mathbf{E }}[\varphi(\mathbf{ Y},\mathbf{ X})]=\hat{\mathbf{ E}}[\hat{\mathbf{ E}}[\varphi(\mathbf{ y}, \mathbf{ X})]|_{\mathbf{ y}=\mathbf{ Y}}].\]
\end {definition}
In a sublinear expectation space,  the fact that $\mathbf{ X}$ is independent from $\mathbf{ Y}$ does not imply that $\mathbf{Y}$ is independent from $\mathbf{X}$. We say $(X_i)_{i\ge1}$ is a sequence of independent random variables means that $X_{i+1}$ is independent from $(X_1,\cdots, X_i)$ for each $i\in\mathbb{N}$.

\begin {definition} Let $(\Omega, \mathcal{H}, \hat{\mathbf{E}})$ and   $(\widetilde{\Omega}, \mathcal{\widetilde{H}}, \widetilde{\mathbf{E}})$ be two sublinear expectations. A random vector $\mathbf{ X}$ in  $(\Omega, \mathcal{H}, \hat{\mathbf{E}})$ is said to be identically distributed with another random vector $\mathbf{ Y}$ in $(\widetilde{\Omega}, \mathcal{\widetilde{H}}, \widetilde{\mathbf{E}})$ (write $\mathbf{ X}\mathop{=}\limits^d\mathbf{ Y}$), if for any bounded and Lipschitz function $\varphi$, \[\hat{\mathbf{ E}}[\varphi(X)]=\widetilde{\mathbf{ E}}[\varphi(Y)].\]

\end {definition}

\section{Regularity Estimates for the $G$-heat Equation}
In this section, we shall introduce  a regularity result for the $G$-heat equation, which is important for us to obtain  the rate of convergence of Peng's Central Limit Theorem.
\begin {eqnarray}\label {Gequation}
u_t(x,t)-G(D^2_xu(x,t))&=&0, \ (x,t)\in \mathbb{R}\times (0,\infty),\\
u(x,0)&=&\varphi(x),
\end {eqnarray}where $G(a)=\frac{1}{2}(\overline{\sigma}^2a^+-\underline{\sigma}^2a^-)$ for some $\overline{\sigma}\ge\underline{\sigma}>0$.

Throughout this paper, we shall always suppose that $\underline{\sigma}>0$.

For regularity estimates of (more general) fully nonlinear partial differential equations, we refer the readers to the papers Kruzhkov (1967), Krylov (1987), Wang (1992) and the book Lieberman (2005) and the references therein. Here we only introduce a result that will be used in this paper.

First of all, for any initial value $\varphi\in C_{b, Lip}(\mathbb{R})$, the
collection of bounded Lipstchiz functions on $\mathbb{R}$, the $G$-heat equation has a unique classical solution. Furthermore, we have the following interior regularity estimates:

\textit{There exists $\alpha\in(0,1)$ depending  on $\underline{\sigma}$ and $\overline{\sigma}$, and a positive constant $c_{\alpha, G}$ depending  on $\alpha, \underline{\sigma}$ and $\overline{\sigma}$
such that if $u\in C^{2,1}(\mathbb{R}\times (0, +\infty))$ is a solution to the $G$-heat equation, we have
\begin {eqnarray}
[D_x^2u(\cdot,1)]_{\alpha}\le c_{\alpha, G}\|Du\|_{\infty, \mathbb{R}\times [0,1]}.
\end {eqnarray}}
Here, $[f]_\alpha=\sup\limits_{x,y\in\mathbb{R}, x\neq y}\frac{|f(x)-f(y)|}{|x-y|^\alpha}$.

Set $v_\varepsilon(x,t)=\frac{1}{\varepsilon}u(\varepsilon x, \varepsilon^2t)$ for $\varepsilon\in(0,1)$. Then $v_\varepsilon$ is also a solution to the $G$-heat equation. So we have
\begin {eqnarray*}
[D_x^2v_\varepsilon(\cdot,1)]_{\alpha}\le c_{\alpha, G}\|Dv_\varepsilon\|_{\infty, \mathbb{R}\times [0,1]}.
\end {eqnarray*}
Noting that
\[[D_x^2 v_\varepsilon(\cdot,1)]_\alpha=\varepsilon^{1+\alpha}[D_x^2u(\cdot,\varepsilon^2)]_\alpha\]
and
\[\|D_x v_\varepsilon\|_{\infty, \mathbb{R}\times [0,1]}\le\|D_xu\|_{\infty, \mathbb{R}\times [0,1]},\]
we get
\[\varepsilon^{1+\alpha}[D_x^2u(\cdot,\varepsilon^2)]_\alpha\le c_{\alpha,G}\|D_xu\|_{\infty, \mathbb{R}\times [0,1]}.\] We summarize the above arguments as the following theorem.

\begin {theorem} \label {IRE-Gequation}  There exists $\alpha\in(0,1)$ depending  on $\underline{\sigma}$ and $\overline{\sigma}$, and a positive constant $c_{\alpha, G}$ depending  on $\alpha, \underline{\sigma}$ and $\overline{\sigma}$
such that if $u\in C^{2,1}(\mathbb{R}\times (0, +\infty))$ is a solution to the $G$-heat equation, we have, for $t\in (0,1]$,
\begin {eqnarray}
[D_x^2u(\cdot,t)]_{\alpha}\le c_{\alpha, G}\frac{1}{t^{\frac{1}{2}+\frac{\alpha}{2}}}\|D_xu\|_{\infty, \mathbb{R}\times [0,1]}.
\end {eqnarray}
\end {theorem}
For $\varphi\in C_{b,Lip}(\mathbb{R})$, if $u$ is the solution to the $G$-heat equation with initial value $\varphi$, we know that $u(\cdot,t)$ is also uniformly Lipschitz continuous with
\[\|D_xu\|_{\infty, \mathbb{R}\times [0,1]}\le |\varphi|_{Lip}.\] Hence, we have the following immediate corollary of Theorem \ref{IRE-Gequation}.
\begin {corollary} There exists $\alpha\in(0,1)$ depending  on $\underline{\sigma}$ and $\overline{\sigma}$, and a positive constant $c_{\alpha, G}$ depending  on $\alpha, \underline{\sigma}$ and $\overline{\sigma}$
such that if  $\varphi\in C_{b,Lip}(\mathbb{R})$ with $|\varphi|_{Lip}\le 1$, and $u$ is the solution to the $G$-heat equation with initial value $\varphi$, then we have
\begin {eqnarray}
[D_x^2u(\cdot,t)]_{\alpha}\le c_{\alpha, G}\frac{1}{t^{\frac{1}{2}+\frac{\alpha}{2}}}.
\end {eqnarray}
\end {corollary}

\section {Rate of Convergence of Peng's CLT}
Let $\mathcal{N}[\varphi]=\sup_{\mu\in\Theta}\mu[\varphi]$ be a sublinear expectation on  $C_{b,Lip}(\mathbb{R})$. Throughout this article, we suppose the following additional property:
\begin{description}
\item[(H)] $\lim_{N\rightarrow\infty}\mathcal{N}[|x|1_{[|x|>N]}]=0.$
\end{description}
Note that the condition $(H)$ is naturally satisfied if $\mathcal{N}[|x|^{1+\delta}]<\infty$ for some $\delta>0$.

Define $\xi: \mathbb{R}\rightarrow\mathbb{R}$ by $\xi(x)=x$. Sometimes, we write $\mathcal{N}_G[\varphi], \ \mathcal{N}[\varphi]$ and  $\mu[\varphi]$ by $\mathbb{E}_G[\varphi(\xi)], \ \mathbb{E}[\varphi(\xi)]$ and $E_\mu[\varphi(\xi)]$, respectively. For  $\varphi\in C_{b,Lip}(\mathbb{R})$, set $\Theta_\varphi=\{\mu\in \Theta:E_{\mu}[\varphi(\xi)]=\mathbb{E}[\varphi(\xi)]\}$.
\begin{lemma}  \label {lemma-SteinEquation}For $\phi\in C_{b,Lip}(\mathbb{R})$, let $v$ be the solution to the $G$-heat equation with initial value $\phi$ and set $\phi_s(x):=v( \sqrt{1-s}x,s)$. Then
\begin{eqnarray}\label {SteinEquation}\mathcal{N}_G[\phi]-\mathcal{N}[\phi]=\int_0^1\frac{1}{1-s}\sup_{\mu_s\in\Theta_s}E_{\mu_s}[\mathcal{L}_G\phi_s(\xi)]ds=\int_0^1\frac{1}{1-s}\inf_{\mu_s\in\Theta_s}E_{\mu_s}[\mathcal{L}_G\phi_s(\xi)]ds,
\end {eqnarray} where $\mathcal{L}_G\phi_s(x)=G(\phi_s''(x))-\frac{x}{2}\phi_s'(x)$, $\Theta_s=\Theta_{\phi_s}$. Particularly, we have, for a.e. $s\in(0,1)$, \[\sup_{\mu_s\in\Theta_s}E_{\mu_s}[\mathcal{L}_G\phi_s(\xi)]=\inf_{\mu_s\in\Theta_s}E_{\mu_s}[\mathcal{L}_G\phi_s(\xi)].\]
\end{lemma}
\begin{proof} Set $w(s)=\mathbb{E}[v( \sqrt{1-s}\xi,s)]$. Then $w(1)=\mathcal{N}_G[\phi]$ and $w(0)=\mathcal{N}[\phi]$.
By Lemma 2.4 in Hu, Peng and Song (2017), we have, for $s\in (0,1)$,
\begin{eqnarray*}
\partial_s^+w(s):&=&\lim_{\delta\rightarrow0+}\frac{w(s+\delta)-w(s)}{\delta}\\
&=&\frac{1}{1-s}\sup_{\mu_s\in\Theta_s}E_{\mu_s}[\mathcal{L}_G\phi_s(\xi)]
\end{eqnarray*}
and
\begin{eqnarray*}
\partial_s^-w(s):&=&\lim_{\delta\rightarrow0+}\frac{w(s-\delta)-w(s)}{-\delta}\\
&=&\frac{1}{1-s}\inf_{\mu_s\in\Theta_s}E_{\mu_s}[\mathcal{L}_G\phi_s(\xi)].
\end {eqnarray*}

Noting
that $w$ is continuous on $[0,1]$ and locally Lipschitz continuous on $(0,1)$ by the regularity properties of the solution $v$ of the $G$-heat equation, we have $w'(s)=\partial_s^+w(s)=\partial_s^-w(s)$  for a.e. $s\in(0,1)$ and consequently \[w(1)-w(0)=\int_0^1\partial_s^+w(s)ds=\int_0^1\partial_s^-w(s)ds.\]
\end{proof}
\begin {remark} \label {remark-SteinE} Suppose that $G(a)=\frac{1}{2}\mathcal{N}_G[ax^2]=\frac{1}{2}\sigma^2 a$ is linear, i.e., $\mathcal{N}_G=N(0,\sigma^2)$, and that $\mathcal{N}$ is a linear expectation, i.e., $\Theta=\{\mu\}$ is a singleton.  Then (\ref{SteinEquation}) can be rewritten as
\begin{eqnarray*}E[\phi(Z)]-E_\mu[\phi(\xi)]=E_\mu[\int_0^1\frac{1}{1-s}\bigg(\frac{\sigma^2}{2}\phi''_s(\xi)-\frac{\xi}{2}\phi'_s(\xi)\bigg)ds]=E_\mu[\frac{\sigma^2}{2}g''(\xi)-\frac{\xi}{2}g'(\xi)],
\end{eqnarray*} where $g(x)=\int_0^1\frac{1}{1-s}\phi_s(x)ds$ and $Z\sim N(0,\sigma^2)$ under $E$. Since this equality holds for any distribution $\mu$, we have, by choosing $\mu=\delta_x$,
\[E[\phi(Z)]-\phi(x)=\frac{\sigma^2}{2}g''(x)-\frac{x}{2}g'(x), \ x\in \mathbb{R},\] which is just the classical Stein Equation. Equ. (\ref{SteinEquation}) will be used as a substitute of the Stein equation under  sublinear expectations.
\end {remark}

The next Lemma gives an estimate of the expectations on the right hand of  Equ. (\ref{SteinEquation}).

\begin {lemma} \label {SteinEstimate} Let $\alpha\in(0,1)$. Suppose $\mathbb{E}[\xi]=\mathbb{E}[-\xi]=0$ and $\mathbb{E}[|\xi|^{2+\alpha}]<\infty$. For $\phi\in C_b^{2,\alpha}(\mathbb{R})$ and $\mu\in\Theta_{\phi}$, we have
\[\bigg|E_{\mu}[\frac{\xi}{2}\phi'(\xi)-G(\phi''(\xi))]\bigg|\le 2 [\phi'']_{\alpha} \mathbb{E}[|\xi|^{2+\alpha}],\] where $G(a)=\frac{1}{2}\mathbb{E}[a|\xi|^2]$, $a\in \mathbb{R}$.
\end {lemma}
\begin {proof} Taylor's  formula gives
\begin{eqnarray}
\label {Taylor1} \phi(\xi)&=&\phi(0)+\phi'(0)\xi+\frac{1}{2}\phi''(0)|\xi|^2+R_{\xi},\\
\label {Taylor2} \phi'(\xi)&=&\phi'(0)+\phi''(0)\xi+R'_{\xi},\\
\label {Taylor3} \phi''(\xi)&=&\phi''(0)+R''_{\xi},
\end{eqnarray}
 with $|R_{\xi}|\le\frac{1}{2}[\phi'']_{\alpha} |\xi|^{2+\alpha},$  $|R'_{\xi}|\le[\phi'']_{\alpha} |\xi|^{1+\alpha}$ and $|R''_{\xi}|\le[\phi'']_{\alpha} |\xi|^{\alpha}.$

  Set $A:=\mathbb{E}[\phi(\xi)]=E_{\mu}[\phi(\xi)]$. Then
\begin {eqnarray*}A=\mathbb{E}[\phi(\xi)]&=&\mathbb{E}[\phi(0)+\phi'(0)\xi+\frac{1}{2}\phi''(0)|\xi|^2+R_{\xi}]\\
&\le& \phi(0)+\mathbb{E}[\frac{1}{2}\phi''(0)|\xi|^2]+\mathbb{E}[R_{\xi}]\\
&\le& \phi(0)+G(\phi''(0))+\frac{1}{2}[\phi'']_{\alpha}\mathbb{E}[|\xi|^{2+\alpha}],
\end {eqnarray*} and
\begin {eqnarray*}A=\mathbb{E}[\phi(\xi)]&=&\mathbb{E}[\phi(0)+\phi'(0)\xi+\frac{1}{2}\phi''(0)|\xi|^2+R_{\xi}]\\
&\ge& \phi(0)+\mathbb{E}[\frac{1}{2}\phi''(0)|\xi|^2]-\mathbb{E}[-R_{\xi}]\\
&\ge& \phi(0)+G(\phi''(0))-\frac{1}{2}[\phi'']_{\alpha}\mathbb{E}[|\xi|^{2+\alpha}].
\end {eqnarray*}
Therefore, \[\bigg|A-\phi(0)-G(\phi''(0))\bigg|\le\frac{1}{2}[\phi'']_{\alpha}\mathbb{E}[|\xi|^{2+\alpha}].\]
Noting that $A=E_{\mu}[\phi(\xi)]=\phi(0)+\frac{1}{2}\phi''(0)E_{\mu}[|\xi|^2]+E_{\mu}[R_{\xi}]$, we have
\begin {eqnarray}
\label {estimate1}
\bigg|\frac{1}{2}\phi''(0)E_{\mu}[|\xi|^2]-G(\phi''(0))\bigg|= \bigg|A-\phi(0)-E_{\mu}[R_{\xi}]-G(\phi''(0))\bigg|\le [\phi'']_{\alpha}\mathbb{E}[|\xi|^{2+\alpha}].
\end {eqnarray}
Now let us compute the expectation $E_{\mu}[\frac{\xi}{2}\phi'(\xi)-G(\phi''(\xi))]$. By (\ref {Taylor2}) and (\ref{Taylor3}), we have
\begin {eqnarray*}
& &\frac{\xi}{2}\phi'(\xi)-G(\phi''(\xi))\\
&=&\frac{\xi}{2}(\phi'(0)+\phi''(0)\xi+R'_{\xi})-G(\phi''(0)+R''_{\xi})\\
&=&\frac{\xi}{2}\phi'(0)+[\frac{1}{2}\phi''(0)|\xi|^2-G(\phi''(0))]+[G(\phi''(0))-G(\phi''(0)+R''_{\xi})]+\frac{\xi}{2} R'_{\xi}.
\end {eqnarray*} So, by (\ref{estimate1}),
\begin {eqnarray*}
& &\bigg|E_{\mu}[\frac{\xi}{2}\phi'(\xi)-G(\phi''(\xi))]\bigg|\\
&=&\bigg|E_{\mu}[\frac{1}{2}\phi''(0)|\xi|^2-G(\phi''(0))]+E_{\mu}[G(\phi''(0))-G(\phi''(0)+R''_{\xi})]+\frac{1}{2}E_{\mu}[\xi R'_{\xi}]\bigg|\\
&\le&[\phi'']_{\alpha}\mathbb{E}[|\xi|^{2+\alpha}]+\frac{1}{2}[\phi'']_{\alpha}\overline{\sigma}^2\mathbb{E}[|\xi|^{\alpha}]+\frac{1}{2}[\phi'']_{\alpha}\mathbb{E}[|\xi|^{2+\alpha}]\\
&\le&2[\phi'']_{\alpha}\mathbb{E}[|\xi|^{2+\alpha}].
\end {eqnarray*}
The last inequality holds since
\[\mathbb{E}[|\xi|^2]\mathbb{E}[|\xi|^{\alpha}]\le\big(\mathbb{E}[|\xi|^{2\times\frac{2+\alpha}{2}}]\big)^{\frac{2}{2+\alpha}}\big(\mathbb{E}[|\xi|^{\alpha\times\frac{2+\alpha}{\alpha}}]\big)^{\frac{\alpha}{2+\alpha}}=\mathbb{E}[|\xi|^{2+\alpha}].\]
\end {proof}
\begin{remark}
We emphasize that the function $G$ in Lemma \ref{lemma-SteinEquation} is determined by the variances of $\mathcal{N}_G$ and that the function $G$ in Lemma \ref{SteinEstimate} is determined by the variances of $\mathcal{N}$. In other words, to estimate $\mathcal{N}_G[\phi]-\mathcal{N}[\phi]$ applying these two lemmas requires that  $\mathcal{N}_G$ and $\mathcal{N}$ have the same variances.
\end{remark}

With these preparations, we are now ready to prove the convergence rate of Peng's Central Limit Theorem under sublinear expectations.

\begin{theorem} \label {Rate-PengCLT} Let $(X_i)_{i\ge 1}$ be a sequence of i.i.d random variables  under a sublinear expectation $\hat{\mathbf{E}}$ with $\hat{\mathbf{E}}[X_1]=\hat{\mathbf{E}}[-X_1]=0$ and $\hat{\mathbf{E}}[X_1^2]=\overline{\sigma}^2\ge-\hat{\mathbf{E}}[-X_1^2]=\underline{\sigma}^2>0$. Setting $W_n:=\frac{X_1+\cdots+X_n}{\sqrt{n}}$, then there exists $\alpha\in(0,1)$  depending  on $\underline{\sigma}$ and $\overline{\sigma}$, and a positive constant $C_{\alpha, G}$ depending  on $\alpha, \underline{\sigma}$ and $\overline{\sigma}$ such that
\[\sup_{|\varphi|_{Lip}\le1}\bigg|\hat{\mathbf{E}}[\varphi(W_n)]-\mathcal{N}_G(\varphi)\bigg|\leq C_{\alpha,G}\frac{\hat{\mathbf{E}}[|X_1|^{2+\alpha}]}{n^{\frac{\alpha}{2}}}ïŒ,\] where $\mathcal{N}_G$ is the $G$-normal distribution with $G(a)=\frac{1}{2}\hat{\mathbf{E}}[aX_1^2]$.
\end{theorem}
Here $\alpha$ is the H$\ddot{\textmd{o}}$lder exponent in Theorem \ref{IRE-Gequation}, and $C_{\alpha,G}$ can be chosen as $\frac{4 }{1-\alpha}c_{\alpha,G}$
with $c_{\alpha,G}$ the $\alpha$-H$\ddot{\textmd{o}}$lder constant in the same theorem.

\begin {proof}
Fix $n\in \mathbb{N}$. Set, for $1\le i\le n$,
\begin{eqnarray*}\xi_{i,n}=\frac{X_i}{\sqrt{n}}, \ W_{0,n}=0, \ W_{i,n}=\sum_{k=1}^i\xi_{k,n}.
\end{eqnarray*}
and, for $0\le i\le n$,
 \[A_{i,n}=\hat{\mathbf{E}}[u(W_{i,n},1-\frac{i}{n})],\] where $u(x,t)$ is the solution to the $G$-heat equation with $u(x,0)=\varphi(x)$.

Then $A_{n,n}=\hat{\mathbf{E}}[\varphi(W_n)]$, $A_{0,n}=\mathcal{N}_G[\varphi]$, and
\begin{eqnarray}
\bigg|\hat{\mathbf{E}}[\varphi(W_n)]-\mathcal{N}_G[\varphi]\bigg|&\le&\sum_{i=1}^n\big|A_{i,n}-A_{i-1,n}\big|\\
&=&\sum_{i=1}^n\bigg|\hat{\mathbf{E}}[b_{i,n}(W_{i-1,n})]-\hat{\mathbf{E}}[c_{i,n}(W_{i-1,n})]\bigg|,\\
&\le& \sum_{i=1}^n\sup_{x\in\mathbb{R}}\big|b_{i,n}(x)-c_{i,n}(x)\big|
\end{eqnarray} where $b_{i,n}(x)=\hat{\mathbf{E}}[u(x+\frac{X_i}{\sqrt{n}}, 1-\frac{i}{n})]$ and $c_{i,n}(x)=\mathbb{E}_G[u(x+\frac{\xi}{\sqrt{n}}, 1-\frac{i}{n})]$. Here and below we write $\mathbb{E}_G[\phi(\xi)]$ for $\mathcal{N}_G[\phi]$.

Let us now compute $b_{i,n}(x)-c_{i,n}(x)$.

Set $\phi(y):=\phi_{x,i,n}(y)=u(x+\frac{y}{\sqrt{n}}, 1-\frac{i}{n})$. Then $c_{i,n}(x)=\mathcal{N}_G[\phi]$ and $b_{i,n}(x)=\hat{\mathbf{ E}}[\phi(X_1)]$. The latter, as a sublinear expectation on $C_{b,Lip}(\mathbb{R})$, can be represented as \[\hat{\mathbf{ E}}[\phi(X_1)]=\sup_{\mu\in\Theta}\mu[\phi],\] where $\Theta$ is a weakly compact subset of probabilities on $\mathbb{R}$. In the sequel, we employ the notations in Lemma \ref{lemma-SteinEquation}.
By this lemma, we have
\[c_{i,n}(x)-b_{i,n}(x)=\int_0^1\frac{1}{1-s}\sup_{\mu_s\in\Theta_s}\mu_s[\mathcal{L}_G\phi_s]ds=\int_0^1\frac{1}{1-s}\inf_{\mu_s\in\Theta_s}\mu_s[\mathcal{L}_G\phi_s]ds,\]
where \begin{eqnarray*}
\phi_s(y)&=&\mathbb{E}_G[\phi(\sqrt{1-s}y+\sqrt{s}\xi)]\\
&=&\mathbb{E}_G[u(x+\sqrt{\frac{1-s}{n}}y+\sqrt{\frac{s}{n}}\xi,1-\frac{i}{n})]\\
&=&u(x+\sqrt{\frac{1-s}{n}}y,1-\frac{i}{n}+\frac{s}{n}).
\end{eqnarray*}
Therefore
\begin{eqnarray*}[D^2_y\phi_s]_{\alpha}&=&(\frac{1-s}{n})^{1+\frac{\alpha}{2}}[D^2_xu(\cdot, 1-\frac{i}{n}+\frac{s}{n})]_{\alpha}\\
&\le& c_{\alpha,G}(\frac{1-s}{n})^{1+\frac{\alpha}{2}}(1-\frac{i}{n}+\frac{s}{n})^{-(\frac{1}{2}+\frac{\alpha}{2})}.
\end {eqnarray*} Now Lemma \ref{SteinEstimate} gives
\begin {eqnarray*}\bigg|b_{i,n}(x)-c_{i,n}(x)\bigg|&\le& \int_0^1\frac{2}{1-s} [D^2_y\phi_s]_{\alpha}ds \times \hat{\mathbf{E}}[|X_1|^{2+\alpha}] \\
&\le& \frac{2c_{\alpha,G}}{n}\int_0^1(\frac{1-s}{n})^{\frac{\alpha}{2}}(1-\frac{i}{n}+\frac{s}{n})^{-(\frac{1}{2}+\frac{\alpha}{2})}ds\times \hat{\mathbf{E}}[|X_1|^{2+\alpha}]\\
&\le&\frac{2c_{\alpha,G}}{n^{1+\frac{\alpha}{2}}}\int_0^1(1-\frac{i}{n}+\frac{s}{n})^{-(\frac{1}{2}+\frac{\alpha}{2})}ds\times \hat{\mathbf{E}}[|X_1|^{2+\alpha}]\\
&=&\frac{2c_{\alpha,G}}{n^{\frac{\alpha}{2}}}\int_{1-\frac{i}{n}}^{1-\frac{i-1}{n}}s^{-(\frac{1}{2}+\frac{\alpha}{2})}ds\times \hat{\mathbf{E}}[|X_1|^{2+\alpha}].
\end {eqnarray*}
Hence,
\begin {eqnarray*}
\bigg|\hat{\mathbf{E}}[\varphi(W_n)]-\mathcal{N}_G[\varphi]\bigg|&\leq&\sum_{i=1}^n\sup_{x\in\mathbb{R}}\big|b_{i,n}(x)-c_{i,n}(x)\big|\\
&\le&\frac{2c_{\alpha,G}}{n^{\frac{\alpha}{2}}}\int_{0}^{1}s^{-(\frac{1}{2}+\frac{\alpha}{2})}ds\times \hat{\mathbf{E}}[|X_1|^{2+\alpha}]\\
&=&\frac{4c_{\alpha,G} }{1-\alpha}\frac{\hat{\mathbf{E}}[|X_1|^{2+\alpha}]}{n^{\frac{\alpha}{2}}}.
\end {eqnarray*}
\end {proof}

\section {Non-identically Distributed Case }

In this section, we consider the normal approximation for an independent but not necessarily identically distributed sequence of random variables. To do so, we first introduce some notations. For a random variable $X$ in a sublinear expectation space with $\overline{\sigma}^2:=\hat{\mathbf{ E}}[X^2]\ge -\hat{\mathbf{ E}}[-X^2]=:\underline{\sigma}^2>0$, set $\beta:=\frac{\overline{\sigma}}{\underline{\sigma}}$ and $\sigma:=\frac{\overline{\sigma}+\underline{\sigma}}{2}$. Now we can use $\beta, \sigma$ to characterize the variances of a  random variable $X$ in a sublinear expectation space. For example, we shall write $\mathcal{N}_\beta(0,\sigma^2)$ for the $G$-normal distribution $\mathcal{N}_G$, and write $\mathcal{N}_\beta$ for $\mathcal{N}_\beta(0,1)$. Clearly, $\mathcal{N}_1(0,\sigma^2)=N(0,\sigma^2)$, the classical normal distribution.

In this section, we shall fix the ratio $\beta\ge 1$ of variances  as a constant and call $\sigma^2$ the variance. We write $G_\beta$ the function $G$ with $\underline{\sigma}=\frac{2}{1+\beta}$ and $\overline{\sigma}=\frac{2\beta}{1+\beta}$. So the $G_\beta$-normal distribution is $\mathcal{N}_\beta$.

\begin{theorem}  Let $(\xi_i)_{1\le i\le n}$ be a sequence of independent random variables  under a sublinear expectation $\hat{\mathbf{E}}$. We suppose further that, for each $1\le i\le n$, $\xi_i$ has finite variance $\sigma_i^2$ and mean 0, i.e.,
$\hat{\mathbf{E}}[\xi_i]=\hat{\mathbf{E}}[-\xi_i]=0$. Setting $W:=\xi_1+\cdots+\xi_n$ and $\sigma^2:=\Sigma_{i=1}^n\sigma_i^2$, then there exists $\alpha\in(0,1)$  depending  on $\beta$, and a positive constant $C_{\alpha, \beta}$ depending  on $\alpha, \beta$ such that
\[\sup_{|\varphi|_{Lip}\le1}\bigg|\hat{\mathbf{E}}[\varphi(\frac{W}{\sigma})]-\mathcal{N}_\beta(\varphi)\bigg|\leq C_{\alpha,\beta}\sup_{1\le i\le n}\bigg\{\frac{\hat{\mathbf{E}}[|\xi_i|^{2+\alpha}]}{\sigma_i^{2+\alpha}}(\frac{\sigma_i}{\sigma})^\alpha\bigg\}.\]
\end{theorem}
Here $\alpha$ is the H$\ddot{\textmd{o}}$lder exponent in Theorem \ref{IRE-Gequation}, and $C_{\alpha,\beta}$ can be chosen as $\frac{4 }{1-\alpha}c_{\alpha,G_\beta}$
with $c_{\alpha,G_\beta}$ the $\alpha$-H$\ddot{\textmd{o}}$lder constant in the same theorem.

\begin {proof} The proof is adapted from that of Theorem \ref {Rate-PengCLT}.

Set, for $1\le i\le n$,
\begin{eqnarray*} t_0=0, \ t_i=\frac{\Sigma_{k=1}^i\sigma_k^2}{\sigma^2}, \ W_0=0,  \ W_{i}=\sum_{k=1}^i\frac{\xi_{k}}{\sigma}.
\end{eqnarray*}
and, for $0\le i\le n$,
 \[A_{i}=\hat{\mathbf{E}}[u(W_{i},1-t_i)],\] where $u(x,t)$ is the solution to the standard $G_\beta$-heat equation with $u(x,0)=\varphi(x)$.

Then $A_{n}=\hat{\mathbf{E}}[\varphi(W_n)]$, $A_{0}=\mathcal{N}_\beta[\varphi]$, and
\begin{eqnarray}
\bigg|\hat{\mathbf{E}}[\varphi(W_n)]-\mathcal{N}_\beta[\varphi]\bigg|&\le&\sum_{i=1}^n\big|A_{i}-A_{i-1}\big|\\
&=&\sum_{i=1}^n\bigg|\hat{\mathbf{E}}[b_{i}(W_{i-1})]-\hat{\mathbf{E}}[c_{i}(W_{i-1})]\bigg|,\\
&\le& \sum_{i=1}^n\sup_{x\in\mathbb{R}}\big|b_{i}(x)-c_{i}(x)\big|
\end{eqnarray} where $b_{i}(x)=\hat{\mathbf{E}}[u(x+\frac{\xi_i}{\sigma}, 1-t_i)]$ and $c_{i}(x)=\mathbb{E}_{\beta}[u(x+\frac{\sigma_i\xi}{\sigma}, 1-t_i)]$. Here and below we write $\mathbb{E}_{\beta}[\phi(\xi)]$ for $\mathcal{N}_{\beta}[\phi]$.

Let us now compute $b_{i}(x)-c_{i}(x)$.

Set $\phi(y):=\phi_{x,i}(y)=u(x+\frac{\sigma_iy}{\sigma}, 1-t_i)$. Then $c_i(x)=\mathcal{N}_\beta[\phi]$ and $b_i(x)=\hat{\mathbf{ E}}[\phi(\frac{\xi_i}{\sigma_i})]$. The latter, as a sublinear expectation on $C_{b,Lip}(\mathbb{R})$, can be represented as \[\hat{\mathbf{ E}}[\phi(\frac{\xi_i}{\sigma_i})]=\sup_{\mu\in\Theta}\mu[\phi],\] where $\Theta$ is a weakly compact subset of probabilities on $\mathbb{R}$.  In the sequel, we employ the notations in Lemma \ref{lemma-SteinEquation}.
By this lemma, we have
\[c_{i}(x)-b_{i}(x)=\int_0^1\frac{1}{1-s}\sup_{\mu_s\in\Theta_s}\mu_s[\mathcal{L}_{G_\beta}\phi_s]ds=\int_0^1\frac{1}{1-s}\inf_{\mu_s\in\Theta_s}\mu_s[\mathcal{L}_{G_\beta}\phi_s]ds,\]
where \begin{eqnarray*}
\phi_s(y)&=&\mathbb{E}_\beta[\phi(\sqrt{1-s}y+\sqrt{s}\xi)]\\
&=&\mathbb{E}_\beta[u(x+\sqrt{1-s}\frac{\sigma_i}{\sigma}y+\sqrt{s}\frac{\sigma_i}{\sigma}\xi,1-t_i)]\\
&=&u(x+\sqrt{1-s}\frac{\sigma_i}{\sigma}y,1-t_i+s\frac{\sigma_i^2}{\sigma^2}).
\end{eqnarray*}
Therefore
\begin{eqnarray*}[D^2_y\phi_s]_{\alpha}&=&(\frac{\sigma_i^2}{\sigma^2}(1-s))^{1+\frac{\alpha}{2}}[D^2_xu(\cdot, 1-t_i+s\frac{\sigma_i^2}{\sigma^2})]_\alpha\\
&\le& c_{\alpha,G_\beta}(\frac{\sigma_i^2}{\sigma^2}(1-s))^{1+\frac{\alpha}{2}}(1-t_i+s\frac{\sigma_i^2}{\sigma^2})^{-(\frac{1}{2}+\frac{\alpha}{2})}.
\end {eqnarray*} Now Lemma \ref{SteinEstimate} gives, noting that $G_\beta(a)=\frac{1}{2}\hat{\mathbf{ E}}[a(\frac{\xi_i}{\sigma_i})^2]$,
\begin {eqnarray*}\bigg|b_{i,n}(x)-c_{i,n}(x)\bigg|&\le& \int_0^1\frac{2}{1-s} [D^2_y\phi_s]_{\alpha}ds \times \frac{\hat{\mathbf{E}}[|\xi_i|^{2+\alpha}]}{\sigma_i^{2+\alpha}} \\
&\le& 2c_{\alpha,G_\beta}\frac{\sigma_i^2}{\sigma^2}\int_0^1(\frac{\sigma_i^2}{\sigma^2}(1-s))^{\frac{\alpha}{2}}(1-t_i+\frac{\sigma_i^2}{\sigma^2}s)^{-(\frac{1}{2}+\frac{\alpha}{2})}ds\times \frac{\hat{\mathbf{E}}[|\xi_i|^{2+\alpha}]}{\sigma_i^{2+\alpha}}\\
&\le&2c_{\alpha,G_\beta}(\frac{\sigma_i^2}{\sigma^2})^{1+\frac{\alpha}{2}}\int_0^1(1-t_i+\frac{\sigma_i^2}{\sigma^2}s)^{-(\frac{1}{2}+\frac{\alpha}{2})}ds\times \frac{\hat{\mathbf{E}}[|\xi_i|^{2+\alpha}]}{\sigma_i^{2+\alpha}}\\
&=&2c_{\alpha,G_\beta}\frac{\sigma_i^\alpha}{\sigma^\alpha}\int_{1-t_i}^{1-t_{i-1}}s^{-(\frac{1}{2}+\frac{\alpha}{2})}ds\times \frac{\hat{\mathbf{E}}[|\xi_i|^{2+\alpha}]}{\sigma_i^{2+\alpha}}.
\end {eqnarray*}

Hence,
\begin {eqnarray*}
\bigg|\hat{\mathbf{E}}[\varphi(W_n)]-\mathcal{N}_\beta[\varphi]\bigg|&\leq&\sum_{i=1}^n\sup_{x\in\mathbb{R}}\big|b_{i,n}(x)-c_{i,n}(x)\big|\\
&\le&2c_{\alpha,G_\beta}\int_{0}^{1}s^{-(\frac{1}{2}+\frac{\alpha}{2})}ds\sup_{1\le i\le n}\bigg\{\frac{\hat{\mathbf{E}}[|\xi_i|^{2+\alpha}]}{\sigma_i^{2+\alpha}}(\frac{\sigma_i}{\sigma})^\alpha\bigg\}\\
&=&\frac{4c_{\alpha,G_\beta} }{1-\alpha}\sup_{1\le i\le n}\bigg\{\frac{\hat{\mathbf{E}}[|\xi_i|^{2+\alpha}]}{\sigma_i^{2+\alpha}}(\frac{\sigma_i}{\sigma})^\alpha\bigg\}.
\end {eqnarray*}
\end {proof}

%%% ----------------------------------------------------------------------

%%%%%%%%%%%%%%%%%%%%%%%âÃ¢Â§âÃ­ââÂ¬â«âÃ­âÃ­âÃ¬âÃ³Ã€
\renewcommand{\refname}{\large References}{\normalsize \ }

\end{document}